\documentclass[a4paper,12pt]{article}

\usepackage{amsmath}
\usepackage{amsrefs}
\usepackage{amsfonts}
\usepackage{amssymb}
\usepackage{amsthm}
\usepackage[cp1250]{inputenc}
\usepackage{setspace}
\usepackage[T1]{fontenc}
\usepackage{geometry}
\usepackage{fancyhdr}
\usepackage{graphicx}
\usepackage{color}
\usepackage{enumerate}
\usepackage{color}

\newcommand{\f}{\widetilde{f}}
\newcommand{\bbE}{{\ensuremath{\mathbb E}}}
\newcommand{\bbN}{{\ensuremath{\mathbb N}}}

\newcommand{\Pro}{{\ensuremath{\mathbb P}}}
\newcommand{\g}{\gamma}

\newcommand{\la}{\lambda}
\newcommand{\bbQ}{{\ensuremath{\mathbb Q}}}
\newcommand{\R}{{\ensuremath{\mathbb R}}}

\newcommand{\bi}{E_{J,n}}
\newcommand{\bll}{E_{J,l}}

\newcommand{\bp}{E_J}
\newcommand{\ai}{A_{J,n}}

\newcommand{\Al}{A_{J2^{l+1}}}
\newcommand{\Bl}{B_{J2^{l+1}}}

\newcommand{\wy}{\frac{1}{1-\gamma}}

\newcommand{\N}{{\ensuremath{\mathbb N}}}

\newcommand{\ee}{e^{n+k}-e^n}

\providecommand{\keywords}[1]{\textbf{\textit{Keywords: }} #1}

\providecommand{\klas}[1]{\textbf{\textit{AMS MSC 2010: }} #1}

\newtheorem{theorem}[subsection]{Theorem}
\newtheorem{lemma}[subsection]{Lemma}
\newtheorem{proposition}[subsection]{Proposition}

\newtheorem{fact}[subsection]{Fact}
\newtheorem{corollary}[subsection]{Corollary}

\newtheorem{remark}[subsection]{Remark}

\begin{document}
\title{Extremal particles in branching processes}

\author{Rafa{\l} Meller\thanks{The research was supported by the Foundation for Polish Science  under the grant HOMING PLUS/2012-6/7.}}
\date{}
\maketitle

\begin{abstract}
The purpose of this study is to investigate the behavior of extremal particle in a spatial branching process on $\mathbb{R}$ with the heavy-tailed compound Poisson process motion and inhomogeneous potential.
\end{abstract}
\klas{60J80}

\section{Introduction}
The problem of examining the position of maximal particle in the
branching processes has drawn a lot of research attention for a long time. The case of constant intensity of birth is quite well understood. A groundbreaking work on this, motivated by biological applications, was done in the 1930's by Kolmogorov-Petrovsky-Piskounov and by Fisher. They considered a model in which particles move accordingly to the Brownian motion and branch with a constant intensity $\la=1$. When a particle branches it is replaced by two independent particles (we will call this model BBM). They showed that if $M(t)$ is a position of top-most-one particle then $M(t)/\sqrt{t} \rightarrow \sqrt{2}$ a.s. Bramson \cite{bram} (see also \cite{Roberts} for a simpler proof) improved that result by showing that a median $m(t)$ of $M$ satisfies $m(t)=\sqrt{2}t-\frac{3}{2\sqrt{2}} \log(t)+O(1)$ as $t \rightarrow \infty$. This is considered to be a classic result. To continue Roberts \cite{Roberts} showed that in BBM we have $\limsup (M(t)-\sqrt{2}t)/\log(t)=-3/(2\sqrt{2})$ and $\liminf (M(t)-\sqrt{2}t)/\log(t)=-1/(2\sqrt{2})$. A modification of BBM was considered in \cite{cat} (local time decides about breeding) and also the nontrivial limit $M(t)/t$ is identified. \newline
Contrary to a constant birth rate, the case of inhomogeneous potential of branching is much less understood. There are only a few papers on that topic. Harris and Harris \cite{HarHar} studied model in which particles move accordingly to the Brownian motion and reproduce with a branching potential $\la(x)=\beta |x|^p$, $\beta>0$, $p \in (0,2]$ (for $p>2$ the model explodes in a finite time). They proved for $p \in (0,2)$ that
$\lim (M(t))/t^{b(p)}=a(p,\beta)$ where $b(p)=2/(2-p)$ and $a(p,\beta)>0$ (they identified the constant $a$). It is also presented that $\lim (\log M(t))/t =\sqrt{2\beta}$ a.s. if $p=2$. A similar model (with a random walk instead of Brownian motion) was considered in \cite{Boch1} and a similar result was obtained.

Two related spatial branching models with inhomogeneous potential are considered in this paper. In the second section a model with particles evolving on $\N$ is introduced (therefore it is called a discrete model). The intensity of branching for a particle having a position $J$ is assumed to be $J^\gamma$ with $\gamma\in(0,1)$. Particles migrate between lines according to a Markov process with geometric tails. This model is just a tool to investigate the process mention in the abstract (it is not interesting itself).
In the second model it is assumed that particles move in $\R$ according to a heavy-tailed compound Poisson process and branch with inhomogeneous breeding potential $(\log x)^\g$, $\g \in (0,1)$. We do not know about any papers studying the position of maximal particle in the branching process with inhomogeneous potential and discontinuous L\'e{}vy motion. We suspect that the methods developed in this paper will be useful to study a generalization of this model in which particles evolve according to an arbitrary heavy-tailed L\'e{}vy process.

\section{Discrete model}
\subsection{\label{model}Model and result} 
We consider a system of  particles with a spatially inhomogeneous branching rate on $\bbN$. The branching process is initialized with a single particle in a position $1$. If a particle has the position $J\in \bbN$ (dubbed as '$J$-th line') it has assigned independent clocks $Z_0,Z_1,Z_2,...$  where $Z_0 \sim Exp(J^\gamma), \ Z_k \sim Exp\left( \beta C^{J+k}\right)$ for $k\geq 1$ [$\beta>0$, $C, \gamma \in (0,1) $ are the parameters of the model and $Exp(\vartheta)$ denotes
the exponential distribution with p.d.f $\vartheta \exp(-\vartheta x)\textbf{1}_{[0,\infty)}(x)$]. 
 Define $T=\min (Z_1,Z_2,...)$. The routine calculations gives $T \sim Exp\left(\frac{ \beta C^{J+1}}{1-C} \right)$. If $T > Z_0$ then the particle produces an additional particle on the $J$-th line after time $Z_0$. If $T=Z_k<Z_0$ then it jumps to the $(J+k)$-th line after time $Z_k$. All particles are independent. After time $T\wedge Z_0$ the particle (or particles if $T>Z_0$ ) receives new independent clocks. 

The objective is to study the displacement of a top-most-one particle.

\begin{remark} Contrary to the well known branching processes the mechanism of jumps depends on the position of a particle. By definition of the model, a particle cannot jump to a lower level than that on which it is. Moreover the parameter $\beta$ will have no impact on the asymptotic of the model. However, this extra parameter will be needed in Section 3, when we will refer to the simpler model.
\end{remark}

Let $X(t)=(X_1(t),X_2(t),\ldots)$ where $X_J(t)$ is a number of particles at time $t$ on the $J$-th line. By definition of the model $\Pro(X_1(0)=1,X_J(0)=0$ for $J=2,3\ldots$)=1. Set 
\begin{equation}
M(t)=\max \left\{J \in \bbN \ | \ X_J(t)>0 \right\} \label{M}
\end{equation}
the position of highest particle at time $t$. 
The following theorem will be shown, which gives information about rough asymptotic of top-most-one particle.
\begin{theorem}\label{th1}
There exist constants $K_1,K_2>0$ such that
$$K_1 \geq \limsup_{t\rightarrow \infty} \frac{M(t)}{ t^{\frac{1}{1-\gamma}}} \geq  \liminf_{t\rightarrow \infty}\frac{M(t)}{ t^{\frac{1}{1-\gamma}}} \geq K_2 \textrm{ a.s.}$$
$K_1=\left(\frac{1}{- \log C}\right)^{\frac{1}{1-\gamma}} $ and $K_2=\left(\frac{2^{1-\gamma}-1}{-4 \log C}\right)^{\frac{1}{1-\gamma}}$ can be taken.
\end{theorem}
We conclude this section with the following simple and very intuitive fact.
\begin{fact}\label{doinf}
The following holds
$$\lim_{t\rightarrow \infty } M(t)= \infty \textrm{ a.s.}$$
\end{fact}
\begin{proof}
Since $M(t)$ is nondecreasing it is enough to show that $M(t)$ converges in probability. The definition of model and the lack of memory property imply that the time a particle need to jump from the $J'$-th line ($J'<J$) to the $J-th$ line or higher is distributed like $Z(J)\sim Exp(\frac{\beta C^J}{1-C})$ (we do not exclude that particle will need to performs several jumps to reach such a line or reproduces before it reaches the $J$-th line). Thus,
$$\Pro\left(M(t)\geq n\right)\geq \Pro(Z(n)\leq t)\rightarrow 1$$
(we look only at the first particle in the model and check whether it reaches $n-th$ line or higher before time $t$).
\end{proof}

\subsection{General facts}
In this section facts used in the main proof are gathered.

Let us denote the law of Galton-Watson process with an intensity of birth $\lambda$ and an intensity of death $\mu$ by $GW(\lambda,\mu)$. Particles always split into two offspring (see chapter 3, page 102 \cite{A-N} for the details). Furthermore,  $GW(\lambda,0)$ is shortcuted to $GW(\lambda)$. For $H(t)$ being a $GW(\lambda,\mu)$ process we define an extinction event
$$
Ext(H)= \{\exists_{t\geq 0} H(t) =0\}.
$$

\begin{proposition}\label{3.0} (\cite{A-N}, pp. 108-111)
Let H(t) be a $GW(\lambda,\mu)$ process. Then its generating function has the following form
\begin{equation}
\bbE s^{H(t)}=f(s,t)=\frac{\mu(s-1)-e^{(\mu-\lambda)t}(\lambda s-\mu)}{\lambda(s-1)-e^{(\mu-\lambda)t}(\lambda s-\mu)}.\label{eq:deathBirthProcess}
\end{equation} Consequently, $\bbE H(t)=e^{(\lambda-\mu)t}$. Moreover, $\Pro(Ext(H(t)))=\min\left\{ \frac{\mu}{\lambda},1 \right\}$. 
\end{proposition}

\begin{theorem}\label{2.3}
Let $H(t)$ be $GW(\lambda, \mu)$ process for $\lambda>\mu>0$. $\{\widetilde{H}(t)\}_{t\geq 0}$ is a number of particles of $H(t)$ which have infinite number of descendants (in particular $\widetilde{H}(0) =0$ if the process $H$ becomes extinct). Then conditionally on $(Ext(H))'$ the process $\widetilde{H}(t)$ has a generating function given by
$$\bbE \left(s^{\widetilde{H}(t)}| (Ext(H))'\right)=\f(s,t)=\frac{se^{-(\lambda-\mu)t}}{1-(1-e^{-(\lambda-\mu)t})s}. $$
\end{theorem}
\begin{proof}
Denote by $\f(s,t)$ the generating function of process in the question. 
The following equation holds (see \cite{A-N} Theorem 1 p. 49 and Theorem 1 p. 110)
$$
\f(s,t)=\frac{f((1-q)s+q,t)-q}{1-q},
$$
where $q$ is the probability that $H(t)$ becomes extinct. By Proposition \ref{3.0} $q=\frac{\mu}{\lambda}$, thus thesis follows by \eqref{eq:deathBirthProcess} and an elementary calculations.
\end{proof}

 For future use, let us state the following Corollary.
\begin{corollary}\label{nier}
Let $Z$ has the geometric distribution with parameter $e^{-\lambda t}$. Then for any $c>0$ the following inequality holds
$$\Pro(Z \ge c \bbE Z ) \ge  \frac{1}{3^c} \textrm{ , for } t\geq \frac{\log 6}{\lambda}.$$
\end{corollary}
\begin{proof}
The following calculations can be made
\begin{align*}
\Pro(Z\ge c \bbE Z)&=\sum_{n=\lceil c \bbE Z \rceil}^{\infty} e^{-\lambda t}\left( 1- e^{-\lambda t} \right)^{n-1}=\left(1-e^{-\lambda t} \right)^{\lceil c \bbE Z \rceil -1}\\
&\ge \left( \left( 1- e^{-\lambda t} \right)^{e^{\lambda t}}\right)^{c}.
\end{align*}
Since the function $g(u)=(1-u)^{\frac{1}{u}}$ is decreasing on $(0,1)$ Corollary follows easily.
\end{proof}

\subsection{Upper bound for the discrete model}\addtocounter{subsection}{-1}  \refstepcounter{subsection}  \label{rozd}

The goal of this chapter is to prove the bound from above in Theorem \ref{th1}. We start with a following simple observation.
\begin{lemma}\label{smierc}
Let $H(t)$ be a $GW(\lambda,\mu)$ process and $H_d(t)$ is a number of deaths in $[0,t]$. Then $\Pro(H_d(t)\geq 1)\leq t \mu e^{ \lambda t}$.
\end{lemma}
\begin{proof}
Let $H^1(t)=H(t)+H_d(t)$. Then $H^1(t)$ is stochastically smaller than $GW(\lambda)$ process. Moreover $ \textbf{1}_{H_d(t)\ge 1} \leq H_d(t)=H^1(t)-H(t)$. Thus
$$\Pro(H_d(t)\geq 1)\leq \bbE H^1(t)-\bbE H(t) \leq e^{\lambda t} - e^{(\lambda - \mu)t}\leq t \mu e^{\lambda t}$$
where second inequality follows by Proposition \ref{3.0}.
\end{proof}

The following lemma is the key technical result of this section.
\begin{lemma}\label{3.1} For an arbitrary $\varepsilon>0$ and $C_1=-(1-\varepsilon)\log C  $ the following series converges
$$\sum_{J=1}^{\infty} \Pro\left(M(C_1 J^{1-\gamma}) \ge J\right) <\infty.$$
\end{lemma}
\begin{proof}
Let us fix $J$. Define processes $\widetilde{X}(t)=(\widetilde{X}_1(t),\widetilde{X}_2(t),\ldots)$ and $\widetilde{M}(t)$  in the same way as $X(t)$ and $M(t)$ (cf. \eqref{M}) but $\Pro(\widetilde{X}_{J-1}(0)=1,\widetilde{X}_{I}(0)=0 \textrm{ for }I\neq J-1)=1$  (we start with a single particle on the $(J-1)$-th line). Since the process $\widetilde{X}(t)$ starts higher than $X(t)$, $\widetilde{M}(t)$ is stochastically larger than $M(t)$, namely
\begin{equation}
\Pro\left(M(t) \ge J\right) \leq \Pro\left(\widetilde{M}(t) \ge J\right). \label{pierwszacegla}
\end{equation}
Consider the process $\widetilde{X}_{J-1}(t)$ and by $(\widetilde{X}_{J-1})_d(t)$ denote a number of deaths in $[0,t]$ (we identify a jump with a death). It is easy to see that $\widetilde{M}(t)\geq J $ if and only if $(\widetilde{X}_{J-1})_d(t) \geq 1$. Because $\widetilde{X}_{J-1}(t)$ is a $GW\left((J-1)^\gamma, \frac{\beta C^J}{1-C}\right)$ process (cf. definition of model in Section \ref{model} )  Lemma \ref{smierc} gives

$$
\Pro\left(\widetilde{M}(t) \ge J\right)=\Pro \left((\widetilde{X}_{J-1})_d(t) \geq 1 \right)\leq t \frac{\beta C^J}{1-C} e^{  (J-1)^\gamma t}. 
$$
Since $C_1=-(1-\varepsilon)\log C $, the above inequality with substituted $t=C_1 J^{1-\gamma}$ and \eqref{pierwszacegla} proves our Lemma.
\end{proof}

We are ready to prove the bound from above in Theorem \ref{th1}. The Borel-Cantelli Lemma and Lemma \ref{3.1} imply that  $M(C_1 J^{1-\gamma})<J$ for $J$ sufficiently large a.s. Hence
\begin{align*} 
\frac{M(t)}{ t^\wy}&\le \frac{J}{\left(C_1 (J-1)^{1-\gamma}\right)^\wy}=\frac{J}{J-1} C_1^{-\wy} 
\end{align*}
for $t \in [C_1(J-1)^{1-\gamma},C_1J^{1-\gamma})$ and $J$ suffieciently large a.s. which means that $\limsup_{t \rightarrow \infty} \frac{M(t)}{t^{\frac{1}{1-\gamma}}} \leq C_1^{-\wy}$ a.s. Since $\varepsilon>0$ was taken arbitrarily the upper bound in Theorem \ref{th1} holds with $K_1=(-\log C)^{-\wy}$.

\subsection{Lower bound for the discrete model}
In this section certain stopping moments, connected with the $J$-th line appear. If $T_J$ is such a moment and it is fixed then we write
\begin{align*}
\Pro_J(\ \cdot \ )= \Pro( \ \cdot \ | T_J,\ \{X_J(T_J)>0 \} ), 
\end{align*}
where $\Pro( A | T_J,\ \{X_J(T_J)>0 \} )=\bbE_{\Pro(\ \cdot \ |  \{X_J(T_J)>0 \})}(\textbf{1}_A | T_J )$ (conditional expectations with respect to the probability measure $\Pro(\ \cdot \ | \ \{X_J(T_J)>0 \})$).
Let us define the following deterministic times 
\begin{equation}
t_J= C_2J^{1-\gamma}-1 \label{t}
\end{equation}
where $C_2=-(4+\varepsilon) \log C$ for arbitrarily chosen $\varepsilon>0$ and $J$ large enough. They are times when typically the number of particles on the $J$-th line is large enough so that one of them performs~a long jump with high probability. To formalize
\begin{equation}
A_J(T_J)=\left\{X_J(T_J)>0,\  X_J( T_J+t_J) \ge e^{C_2 \frac{J}{2}} \right\} \label{zb} 
\end{equation} 
where $T_J$ is a stopping time. 
We shortcut $A_J(T_J)$ to $A_J$ when it does not lead to misunderstanding.

 We have the following Proposition.
\begin{proposition}\label{2.6} 
Assume that $T_J$ is a stopping time that satisfies the following condition
\begin{equation}
X_J(T_J)>0 \textrm{ on } \{T_J < \infty \}. \label{wlasnosc}
\end{equation}
Then for $J$ large enough
\[
	\Pro_J(A_J) \geq 1 - \frac{2}{J^2}.
\] 
\end{proposition}
\begin{remark}
Formally $\Pro_J(A_J)$ is a random variable, however  it is bounded from below by a deterministic value.
\end{remark}

\begin{proof}
Let us choose any particle on the $J$-th line which has position $J$ at time $T_J$.
Let $X^1_J(T_J+t)$ be a number of particles on the $J-th$ line whose ancestor was that particle. By definition of the model (cf. Section \ref{model}) and the strong Markov property (with respect to the stopping time $T_J$), conditionally on $T_J$, $\{ T_J< \infty \}$ the
process $X^1_J(T_J+t)$ is distributed as $GW(J^\gamma,\beta C^{J+1}/(1-C))$ (at $T_J$ the particle on $J$-th line forgets its history and starts an independent life). 
Define $(Ext(X^1_J))'=\{ X^1_J(T_J+t)>0 \textrm{ for all } t>0 \}$. $\widetilde{X^1_J}(T_J+t)$ is the number of those particles among $X^1_J(T_J+t)$ which have an infinite number of descendents. Since $X_J(T_J+t)\geq \widetilde{X^1_J}(T_J+t)$ the following holds
\begin{align}
\Pro_J(A_J)&\geq \Pro_J ( \widetilde{X^1_J}(T_J+t) \geq e^{C_1 J/2} ) \nonumber \\
&=\Pro_J ( \widetilde{X^1_J}(T_J+t) \geq e^{C_1 J/2} | (Ext(X^1_J))' ) \Pro_J ((Ext(X^1_J))' ).\label{qer}
\end{align}

By Theorem \ref{2.3}  $\widetilde{X^1_J}(T_J+t)$ has, conditionally on
$(Ext(X^1_J))'$, the geometric distribution with a  parameter $e^{-(J^\gamma-\beta C^{J+1}/(1-C) )t}$. Thus its expectation equals to
 $$\bbE_{\Pro_J} \left(\widetilde{X^1_J}(T_J+t_J) \ | \ (Ext(X^1_J))'\right)=e^{\left(C_2 J^{1-\gamma}-1\right)\left(J^\gamma-\frac{\beta C^{J+1}}{1-C} \right)}.$$
So for large enough $J$ we have $\bbE_{\Pro_J} \left(\widetilde{X^1_J}(T_J+t_J) \ | \ (Ext(X^1_J))'\right) \geq J^2 e^{C_2 J/2}$. Therefore, using Corollary \ref{nier} for $J$ large enough,  we get 
\begin{align*}
&\Pro_J\left(\widetilde{X^1_J}(T_J+t_J) \ge e^{C_2 J/2} \ | \ (Ext(X^1_J))'  \right)\\
&\geq \Pro_J\left(\widetilde{X^1_J}(T_J+t_J) \ge \frac{ \bbE_{\Pro_J} (\widetilde{X^1_J}(T_J+t_J)\ | \ (Ext(X^1_J))')}{J^2} \ | \ (Ext(X^1_J))' \right) 
\ge \left( \frac{1}{3} \right)^{\frac{1}{J^2}}. 
\end{align*}
Moreover, by  Proposition \ref{3.0}, for large enough $J$  
$$\Pro_J\left(Ext(X^1_J)\right)=\frac{\frac{\beta C^{J+1}}{1-C}}{J^\gamma} \le C^{J+1}.$$
Hence, from \eqref{qer} $\Pro_J(A_J) \ge \left( \frac{1}{3} \right)^{\frac{1}{J^2}} \left( 1-C^{J+1} \right)$.
The statement of proposition follows now by elementary considerations.
\end{proof}

If $T_J$ is a stopping time then we define a random interval of time
$$I_J(T_J)=\begin{cases}  [T_J+t_J,T_J+t_J+1] &\textrm{ if } X_J(T_J)>0,\ T_J<\infty \\ \emptyset &\textrm{ otherwise} \end{cases}  $$
and
\begin{align}
\alpha_J(T_J)=\Pro(\textrm{a single particle being on the } J-th \textrm{ line at time } T_J+t_J    \nonumber \\
\textrm{will jump to the } 2J \textrm{-th line}\textrm{ in } I_J(T_J) |T_J,\ \{T_J<\infty\},\{X_J(T_J+t_J) > 0\}). \label{a}
\end{align}
In the definition of $\alpha_J(T_J)$ it is not excluded that before a jump a particle can produce some additional particles. If $T_J$ is fixed then  $\alpha_J(T_J)$ is shortened to $\alpha_J$.
\begin{proposition}\label{2.7}
If $J$ is sufficiently large and $T_J$ is a stopping time then $\alpha_J \geq \frac{\beta C^{2J}}{2}$.
\end{proposition}
\begin{proof} Let $Z_1,Z_2,\ldots$ be independent variables such that $Z_k \sim Exp(\beta C^{J+k})$ (cf. the definition of model in Section \ref{model}). Define $\widetilde{Z}=\inf_{k \in \bbN \setminus \{J\}} Z_k $. By calculating the CDF, we get that $\widetilde{Z} \sim Exp(\beta C^{J+1}/(1-C) - \beta C^{2J})$. By definition of the model and the strong Markov property (with respect to the stopping time $T_J+t_J$) we have 
$$\alpha_J \geq \Pro(\widetilde{Z} \geq 1 \geq Z_J)\geq \frac{2}{3} \left(1-\exp(-\beta C^{2J})  \right) \geq \frac{\beta}{2} C^{2J}$$
where two last inequalities are true for large enough $J$.
\end{proof}

To continue let us define 
\begin{align}
B_J(T_J)=\{&T_J<\infty, X_J(T_J)>0,\nonumber \\ 
&\textrm{ in } I_J (T_J)\textrm{ occurs a jump from the } J\textrm{-th to the } 2J \textrm{-th line} \},\label{Bj}
\end{align}
where $T_J$ is a stopping time, $t_J$ is defined in (\ref{t}). As usual $B_J(T_J)$ is shortened to $B_J$ if this does not lead to misunderstanding.

\begin{proposition}\label{2.9} Assume that $T_J$ is a stopping time, which satisfies \eqref{wlasnosc}. Then the following estimate holds for large enough $J$
$$ \Pro_J(B_J\ | \ A_J)\geq 1 - e^{-J}.$$
\end{proposition}
As a direct corollary of this fact and Proposition \ref{2.6} we obtain that for large enough~$J$ 
\begin{equation}
	\Pro_J(A_J \cap B_J ) \geq \left(1-\frac{1}{ J^2}\right) \left(1-e^{-J}\right) \geq 1-\frac{4} {J^2}.\label{eqn:fundamentalEstimate}
\end{equation}

\begin{proof}
Since  $A_J$ implies (see (\ref{zb})) $X_J(T_J+t_J) \geq  e^{C_2 J/2} $ and recall \eqref{a} we see that
\begin{align}
1-\Pro_J(B_J \ |  A_J)\leq (1-\alpha_J)^{ e^{C_2 J/2} }=\left( (1-\alpha_J)^{\frac{1}{\alpha_J}} \right)^{\alpha_J  e^{C_2 J/2} } \leq e^{-\alpha_J  e^{C_2 J/2}}. \label{loc111} 
\end{align}
Proposition \ref{2.7} implies for large enough $J$ (because $C_2=-(4+\varepsilon)\log C$)  
\begin{equation}
\alpha_J e^{C_2 J/2}   \ge \frac{\beta}{2}C^{2J}  e^{C_2 J/2} >J. \label{wzor1}
\end{equation}

Combining \eqref{loc111} and \eqref{wzor1} we obtain
$$1-\Pro_J(B_J \ |  A_J)\leq \left(\frac{1}{e} \right)^{J}$$
what implies the Proposition.
\end{proof}

Set $\bp(T_J)=\bigcap_{l=0}^{\infty}B_{J2^l}(T_J+t_{J,l})$ where $t_{J,l}=(t_J+1)+\ldots+(t_{J2^{l-1}}+1)$, $t_{J,0}=0$. The following proposition gives a reason to the mentioned above definition.

\begin{proposition}\label{2.10} If $T_J$ is a stopping time which fulfills \eqref{wlasnosc} then for any $J$ 
\begin{equation}
\bp(T_J) \subset \left\{ \liminf_{t\rightarrow \infty} \frac{ M(t) }{ t^{\frac{1}{1-\gamma}}} \ge \ \left(\frac{2^{(1-\gamma)}-1}{(1+\varepsilon)C_2} \right)^\wy \right\}=S. \label{S}
\end{equation}
\end{proposition}
\begin{proof} Fix $\omega \in \bp$. 
We choose arbitrary $t>T_J(\omega)$ ( $\omega \in B_J(T_J)$ what ensures that $T_J(\omega) <\infty$) and pick largest $m$ such that:
$$T_J(\omega)+C_2\sum_{l=0}^{m-1} \left(J2^l \right)^{1-\gamma}=T_J(\omega)+t_{J,m}<t.$$

Since $\omega \in B_{J2^m}(T_J+t_{J,m})$ we have that $M(t,\omega) \geq J2^m$ thus:
\begin{equation}
 \frac{ M(t,\omega) }{ t^{\frac{1}{1-\gamma}}} \ge  \frac{ J2^{m}}{ \left(T_J(\omega)+C_2\sum_{l=0}^{m} \left(J2^l  \right)^{1-\gamma}\right)^{\frac{1}{1-\gamma}}}. \label{aaaa}
\end{equation}
For large enough $m$ the right hand side in (\ref{aaaa}) is bigger than
\begin{align*}
&\frac{2^{m}}{\left((1+\varepsilon)C_2 \sum_{l=0}^m (2^l)^{1-\gamma}\right)^\wy}\geq \left(\frac{2^{(1-\gamma)}-1}{(1+\varepsilon)C_2} \right)^\wy.
\end{align*}  
So $\omega \in S $ and the proper inclusion holds.
\end{proof}

From the above proposition it is important to establish bounds for the probability of $\bp(T_J)$. The following lemma will be a crucial ingredient of proof of Theorem~\ref{th1}

\begin{lemma}\label{2.13} If $J$ is large enough and $T_J$ fullfils \eqref{wlasnosc} then
$$\Pro_J( \bp(T_J)  ) \geq e^{-\frac{32}{3J^2}}.$$
\end{lemma}
\begin{proof} Additionally denote  
$$\ai(T_J) =\bigcap_{l=0}^{n} A_{J2^l}(T_J+t_{J,l}),\  \bi(T_J) =\bigcap_{l=0}^{n} B_{J2^l}(T_J+t_{J,l}).$$ 
Observe that for $l\leq n$ $\bll(T_J)$ implies $X_{2^lJ}(T_J+t_{J,l})>0$, thus by \eqref{eqn:fundamentalEstimate} and the strong Markov property (with respect to the stopping time $T_J+t_{J,l}$) the following holds
\begin{equation}\label{wl1}
	\Pro_J(\Bl(T_J+t_{J,l}) \cap \Al(T_J+t_{J,l})| A_{J,l}(T_J) \cap \bll(T_J) ) \geq \left(1-\frac{4}{(J2^{l+1})^2}\right).
\end{equation}
Applying the chain rule for conditional probability to \eqref{wl1} and logarithming both sides the following inequality is obtained for $J$ sufficiently large
\begin{equation}\label{aa}
	\log  \Pro_J(\bi(T_J) \cap \ai(T_J)  ) \geq \sum_{l=0}^{n} \log\left(1-\frac{4}{(J2^l)^2}\right) \geq -\sum_{l=0}^{n} \frac{8}{(J2^l)^2},
\end{equation} 
where the last estimate follows by an elementary inequality $2t\leq \log(1+t)$ valid for $t<0$ sufficiently close to $0$.
By the continuity of probability and \eqref{aa} we get
\[
	\Pro_J(\bp(T_J)  )\geq \lim_{n\rightarrow \infty }\exp\left( -\sum_{l=0}^n \frac{8}{(J2^l)^2}  \right).
\]
The thesis follows by an elementary consideration.
\end{proof}

Now we are ready to prove the bound from the below in Theorem \ref{th1}. For a fixed $K\in \bbN$ define the following stopping time
\begin{equation}
T_K=\inf \{t> 0\ | \ \exists_{J\geq K} \ X_J(t)>0 \}. \label{Y}
\end{equation}
Fact \ref{doinf} implies that $T_K$ is well defined.
Recall $S$ defined in \eqref{S}. Proposition \ref{2.10} implies that
\begin{align*}
\Pro(S)&=\sum_{J=K}^\infty \Pro(S|X_J(T_K)>0)\Pro(X_J(T_K)>0)\\
&\geq\sum_{J=K}^\infty \bbE \Pro(E_J(T_K)|T_K,\{T_K<\infty,X_J(T_K)>0\}) \Pro(X_J(T_K)>0)\\
&\geq e^{-\frac{32}{3K^2}}\xrightarrow{K\rightarrow \infty} 1
\end{align*}
where the last inequality follows by Lemma \ref{2.13}.
Thus $\Pro(S)=1$. Since $\varepsilon>0$ was arbitrary, the lower bound in Theorem \ref{th1} is true.




\section{Branching model with heavy-tailed Poisson process and inhomogeneous potential of branching} 
\subsection{Model and notation} \addtocounter{subsection}{-1} \refstepcounter{subsection}  \label{aaa}

In this chapter a system of branching process in inhomogeneous breeding potential on $\mathbb{R}$ is considered. $\nu$ is fixed probability measure on $\R$ which satisfies two conditions:
\begin{align}
\nu\left((-\infty,0)\right) &=0 \label{nieujem}
\end{align}
and
\begin{align}
\lim_{x \rightarrow \infty} \frac{\nu\left([x,+\infty)\right)}{x^{-\alpha}L(x)} &=1 \  \label{asym}
\end{align}
for some $\alpha>0$ (a parameter of model) and a slowly varying function $L:(0,+\infty)\mapsto (0,+\infty)$.

  At time $t=0$ the system is initialized by single particle located at position $x=e$ (the system can be initialized with one particle at $0$, but the proof would be more technical). If at time $\tau$ a particle has a position $Y(\tau)$, then after time $h>0$ it has a position $Y(\tau+h)=Y(\tau)+\sum_{i \leq N(h)} \xi_i$ where $N(h)$ is Poisson process with an intensity $\lambda>0$ (a parameter of model), $(\xi_i)_{i \in \bbN}$ are i.i.d with a distribution $\nu$ ($\Pro(\xi_i \in A)=\nu(A)$). Also $N(h)$, $(\xi_i)_{i \in \bbN}$ and our process till time $t$ are independent. Observe that from \eqref{nieujem} it follows that a trajectory of a particle is a nondecreasing function.

\begin{remark}
Assume that $Z(t)=\sum_{i \leq N(t)} \xi_i$. It is an easy exercise to show that $Z(t+h)-Z(t)$ and $Z(h)$ have the same distribution. Therefore the above definition of the motion in our model is correct.
\end{remark}


The mechanism of branching is described as follows. Assume that at time $\tau$ a new particle $r$ is created. This particle receives a clock  $\mathcal{E}$, distributed like $Exp(1)$, which is independent of the model. Let us assume that the trajectory of $r$ is $Y(t)$. At time $\tau+T$ where
\begin{align}
T=\inf \left\{h>0\ | \ \int_{\tau}^{\tau+h} \log^\gamma(Y(u)) du = \mathcal{E}  \right\}, \label{ga}
\end{align}
where $\gamma \in (0,1)$ is a parameter of the model responsible for branching mechanism, $r$ produces one descendant at its location, which execute the same dynamics. Also one should take $\log^\gamma(x+1)$ if the system is initialized at $0$. At time $\tau+T$ all particles receive new independent clocks. All particles are independent.

Our process formally takes values in the space of all point measures. Denote by $X(t)$ a set of all particles at time $t$ (we will refer to $X(t)$ as our process). So $|X(t)|$ is the number of particles at time $t$ ($|A|$ is a cardinality of set A). For any $r \in X(t)$ by $Y_r(t)$ denote a position of $r$ at time $t$. We sometimes shortcut it to $Y(t)$ if $r$ is fixed. Let
\begin{align} 
M(t)=\sup\{Y_r(t) \ | \ r \in X(t) \} \label{M1} 
\end{align}
be a position of top-most-one particle.
The main goal of this chapter is the following theorem.
\begin{theorem}\label{glowne} 
There exist constants $K_3,K_4>0$ such that
$$K_3\geq\limsup_{t \rightarrow \infty} \frac{\log M(t)}{t^{\wy}}\ge 
\liminf_{t \rightarrow \infty} \frac{\log M(t)}{t^{\wy}} \geq K_4 \ a.s.$$
It can be taken $K_3=(\frac{2}{\alpha})^\wy$ and $K_4=(\frac{2^{1-\gamma}-1}{4D\alpha})^\wy$ where $D$ is a constant which depends only on a measure $\nu$ and is introduced in Lemma \ref{miara}.
\end{theorem} 
We start with two technical lemmas. The first one follows directly by the Karamata representation theorem (see  \cite[Corollary 2.1]{Res}).

\begin{lemma} \label{wolnozmieniajace}
Let L be a slowly varying function. Then for every $\eta>0$ and large enough $x$ we have $L(x)\le x^{\eta}$ and $L(x)\geq x^{-\eta}$.
\end{lemma} 

\begin{lemma}\label{miara} 
There exists a constant $D=D(\nu)>0$ such that for any $n,k \in \bbN$ the following inequality holds
$$\nu\left((e^{n+k}-e^n,\infty)\right) \ge e^{-(n+k)D\alpha}$$
where $\alpha$ is the same as in (\ref{asym}).
\end{lemma}
\begin{proof}
For an arbitrary $\eta>0$  (\ref{asym}) and Lemma \ref{wolnozmieniajace} implies
$$\nu\left((e^{n+k}-e^n,\infty)\right) \geq (1-\varepsilon)(\ee)^{-\alpha}(\ee)^{-\eta}$$
for large enough $n+k$. For the rest of $n,k$ (there are only finite number of them) the above measure is bounded away from $0$. So there exists a constant $\bar{C}$ such that 
$$\nu\left((e^{n+k}-e^n,\infty)\right) \geq \bar{C} \cdot (\ee)^{-\alpha-\eta}$$
for any $n,k \in \bbN$. Now the proof follows by a simple calculations.
\end{proof}

\subsection{Bound from below} The proof strategy is to map the model to the one studied in the previous section. To this end we divide particles into groups $X_J(t)=\{ r \in X(t)\ | \ Y_r(t) \in [e^J,e^{J+1}) \}$. The intuition is that a particle in $X_J$ behaves similarly to a particle on the $J$-th line in the model investigated in the previous chapter. Let us start with two propositions concerning the intensity of reproduction and jumping in our model. In Propositions \ref{rozmnazanie} and \ref{skoki} we fix $t_0>0$ and denote $\bbQ_J(\ \cdot \ )= \Pro(\ \cdot \ |X_J(t_0), \{|X_J(t_0)|>0\})$.
\begin{proposition}\label{rozmnazanie}
Assume that $r\in X_J(t_0)$, and let $T$ be time which it needs to split. Then $T$ is stochastically smaller than $Exp\left(J^\gamma\right)$.
\end{proposition}
\begin{proof}
Denote the position of $r$ by $Y(t)$. From the assumption $Y(t_0) \geq e^J$. By \eqref{nieujem} $Y(t)$ is nondecreasing function so
$$\int_{t_0}^{t_0+h} \log^\gamma(Y(s)) ds \ge \int_{t_0}^{t_0+h} \log^\gamma(e^J) ds= hJ^\gamma. $$
Let us take $\mathcal{E}\stackrel{d}{=}Exp(1)$ which is independent of our process. We have that (cf. \eqref{ga})
$$\bbQ_J\left(T \leq h \right)=\bbQ_J\left(\int_{t_0}^{t_0+h}\log^\gamma(Y(s)) ds>\mathcal{E} \right) \ge \bbQ_J\left(hJ^\gamma >\mathcal{E}\right)=1-e^{-hJ^\gamma}.$$ \hspace{0.001cm}
\end{proof}

\begin{proposition} \label{skoki}
Assume that $r \in X_J(t_0)$. Take $T=\inf\{h\in \R_+ \ | \ Y(t_0+h) \geq e^{J+k}  \}=\inf\{ h \in \R_+ \ | \ r \in \bigcup_{m=k} X_{J+m}(t_0+h)\}$. Then $T$ is stochastically smaller  than $Exp\left(e^{-(J+k)D\alpha}\right)$ and $D$ is a constant which depends only on the jump measure $\nu$.
\end{proposition}
\begin{proof}
By assumptions $Y(t_0) \geq e^J$. Let us decompose the process $h\mapsto Y(t_0+h)-Y(t_0)$ into 
\begin{align*}
Y^1(h)=&\sum_{i< N(h)} \xi_i \textbf{1}_{\xi_i \geq e^{J+k}-e^J}\\
\intertext{ and }
Y^2(h)=&\sum_{i\leq N(h)} \xi_i \textbf{1}_{\xi_i \leq e^{J+k}-e^J}.
\end{align*}
From \eqref{nieujem} the process $Y^2(h)$ is non-negative. Thus if $Y^1(h) \geq e^{J+k} - e^J$ then $Y(t_0+h)\geq e^{J+k}$. 

\begin{align}
\bbQ_J &\left( Y(t_0+h) \geq e^{J+k}  \right) \ge \bbQ_J \left( Y^1(h)\geq e^{J+k} - e^J \right)=\bbQ_J(Y^1(h)>0). \label{box} 
\end{align}
Elementary calculations give $\inf \{h \in \R_+ \ | \ Y^1(h)>0 \} $ is distributed like \\
$Exp\left(\lambda \nu(e^{J+k}-e^J,\infty)\right)$, which by Lemma~\ref{miara} is stochastically smaller than \newline
$Exp\left(\lambda e^{-(J+k)D\alpha}\right)$. 
\end{proof}

We are ready to prove the lower bound in Theorem \ref{glowne}. Define $$\widetilde{X}(t)=(|X_1(t)|,|X_2(t)|,\ldots).$$
 If $r \in X_J(t)$ then we say that $r$ is on the $J$-th line in the $\widetilde{X}(t)$ model. Denote $\widetilde{M}(t)=\sup\{J \in \bbN \ | \ |X_J(t)|>0 \}$.

 It is obvious that
\begin{equation}
 \log M(t) \geq \widetilde{M}(t) \ a.s \label{oczywistosc}
\end{equation}
 (cf. \eqref{M1}).

Let us consider the model introduced in Section $2$ with $\beta=\lambda (1-e^{-D\alpha})$, $C=e^{-D\alpha}$, a birth intense parameter $\gamma$ (the same as in \eqref{ga})  and name it $\bar{X}(t)=(\bar{X}_1(t),\bar{X}_2(t),\ldots)$ (let us note that for any $J\in \bbN$, $\bar{X}_J(t)$ is a number while $X_J(t)$ is a set). Take $\bar{M}(t)=\sup \{J\in \bbN \ | \ \bar{X}_J(t)>0 \}$. 

By Propositions \ref{rozmnazanie} a particle on the $J$-th line in the $\widetilde{X}(t)$ model reproduces faster than a particle on the $J$-th line in the  $\bar{X}(t)$ model (discrete model). Moreover, Proposition \ref{skoki} implies that a particle on the $J$-th line in the $\widetilde{X}(t)$ model needs stochastically less time to reach $J+k$-th line or higher than a particle on the $J$-th line in the $\bar{X}(t)$ model to  reach $J+k$-th line or higher.

We conclude these observations as follows: $\widetilde{M}(t)$ is stochastically larger than $\bar{M}(t)$. That combined with \eqref{oczywistosc} gives
$$\liminf_{t \rightarrow \infty} \frac{\log(M(t))}{t^{\wy}} \stackrel{a.s}{\geq} \liminf_{t \rightarrow \infty}\frac{\widetilde{M}(t)}{t^{\wy}} \stackrel{a.s}{\geq} \liminf_{t \rightarrow \infty} \frac{\bar{M}(t)}{t^{\wy}} \stackrel{a.s}{\geq} K_4. $$
Last inequality comes from Theorem \ref{th1}.

\subsection{Bound from above}
The aim of this section is to prove the bound from above in Theorem \ref{glowne}. The strategy of proof is analogical to that in Subsection  \ref{rozd}. Before the counterpart of Lemma \ref{3.1} will be formulated we need an additional lemma in which a trajectory of the first particle in our model is studied. Define $s_J=J^{1-\gamma} (\frac{\alpha}{2}-\varepsilon)$ where $\varepsilon>0$ is arbitrary.

\begin{lemma} \label{ruch}
Let $Y(t)=e+\sum_{i\leq N(t)}\xi_i$ where $N(t)$ is the Poisson process, with intensity $\lambda$, and $(\xi_i)_{i\in \bbN}$ are i.i.d random variables, independent of $N(t)$, distributed like $\nu$. Then for every $\eta, q>0$ there exists $J_0=J_0(q,\eta)$ such that for $J>J_0$
$$\Pro(Y(s_J)>e^J)\leq e^{-J(\alpha-\eta)}\cdot J^{1+\alpha}+q^{J}.$$
\end{lemma}

\begin{proof}
The following simple calculations are true for any $J\geq 2$
\begin{align}
&\Pro(Y(s_J)>e^J) = \sum_{i=1}^{\infty} \Pro\left(Y(s_J)>e^J \ | N(s_J)=i \right) \Pro\left( N(s_J)=i\right) \nonumber \\
&=\sum_{i=1}^{\infty} \Pro \left( \sum_{j=1}^i \ \xi_j >e^J-e\right) e^{-s_J \lambda} \frac{(s_J \lambda)^i}{i!} \le \sum_{i=1}^\infty \Pro\left( \ \exists_{j\le i} \  \xi_j > \frac{e^J}{2i}\right)e^{-\lambda s_J} \frac{(s_J \lambda)^i}{i!} \nonumber \\
&\le  \sum_{i=1}^J i \Pro\left(\xi_1 \ge \frac{e^J}{2i}\right)e^{-\lambda s_J} \frac{(s_J \lambda)^i}{i!} + \sum_{i=J+1}^\infty e^{-\lambda s_J} \frac{(s_J \lambda)^i}{i!}=S_1+S_2. \label{S_1+S_2}
\end{align}
For large enough $J$, $S_2$ can be estimated using the Markov inequality with test function $e^{x}$ to obtain
\begin{equation}
S_2 \leq e^{(e-1)\lambda s_J-J}\leq q^J \label{drugikaw}
\end{equation}
where $q$ is a constant which can be made arbitrary close to $0$ since  $s_J=J^{1-\gamma} (\frac{\alpha}{2}-\varepsilon)$.

Now we estimate $S_1$.
\begin{equation}
S_1=\sum_{i=1}^J i \ \frac{\nu\left((\frac{e^J}{2i},\infty )\right)}{\lambda} e^{-\lambda s_J} \frac{(s_J \lambda)^i}{i!} \le 
\sum_{i=1}^J \frac{2i}{\lambda} \left(\frac{e^J}{2i} \right)^{-\alpha} L\left(\frac{e^J}{2i} \right)  e^{-\lambda s_J} \frac{(s_J \lambda)^i}{i!} \nonumber
\end{equation}
where (\ref{asym}) for large enough $J$ is used in the last inequality. Now by Lemma \ref{wolnozmieniajace} applied to a slowly varying function $\frac{2}{\lambda} L(x)$ the following holds for arbitrary $\eta>0$ and large enough $J$

\begin{align}
S_1 &\le  \sum_{i=1}^J i \left(\frac{e^J}{2i} \right)^{-\alpha+\eta} e^{-\lambda s_J} \frac{(s_J \lambda)^i}{i!} \le\frac{ e^{-J(\alpha-\eta)}\cdot J^{\alpha-\eta+1}}{2^{-\alpha+\eta}} \sum_{i=1}^J  e^{-\lambda s_J} \frac{(s_J \lambda)^i}{i!} \nonumber \\
&\le e^{-J(\alpha-\eta)}\cdot J^{1+\alpha}. \label{pierwszykaw}
\end{align}
 The inequalities \eqref{S_1+S_2}, \eqref{pierwszykaw} and \eqref{drugikaw} imply the thesis.
\end{proof}

\begin{lemma} \label{ostatnie}
The following series converges $\sum \Pro\left(M(s_J) > e^J\right) < \infty$.
\end{lemma}

\begin{proof}
Fix $J \in \bbN$. Let us consider a branching process  $\widetilde{X}(t)$  defined in the same way as $X(t)$ in Subsection \ref{aaa}, except that any particle above level $e^J$ does not breed. Formally if $r \in \widetilde{X}(t)$ and $T=\inf\{s\in R_+\ | \ s>t, \ Y_r(s)>e^J \}$ then $r$ stops reproducing at time $T$. 

Define $\widetilde{M}(t)=\sup \{ Y_r(t) \ | \ r\in \widetilde{X}(t) \}$.
It is obvious that
\begin{equation}
\Pro\left(M(s_J)>e^J\right) = \Pro \left(\widetilde{M}(s_J) >e^J \right). \label{5.1}
\end{equation}

Let us write
\begin{align}
\Pro\left( \widetilde{M}(s_J) >e^J \right) &\le \Pro\left(\widetilde{M}(s_J) >e^J, \ |\widetilde{X}(s_J)| < \lceil e^{\frac{\alpha}{2}J} \rceil \right) + \Pro\left(|\widetilde{X}(s_J)| \ge \lceil e^{\frac{\alpha}{2}J} \rceil\right) \nonumber \\
&=I_1+I_2. \label{suma}
\end{align}
Analogously as in Proposition \ref{rozmnazanie} it can be shown that if $r\in \widetilde{X}(t)$ then time needed for $r$ to reproduce is stochastically larger than $Exp(J^\gamma)$ (if $Y_r(t)>e^J$ then this time is infinite, which is still stochastically larger then $Exp(J^\gamma)$). Let $H(t)$ be the $GW(J^\gamma)$ process. The previous observation implies that  $|\widetilde{X}(t)|\stackrel{st}{\leq} H(t)$. 
So by the Markov inequality and Proposition \ref{3.0}
\begin{equation} \label{oszacowanieI2}
I_2 \le \frac{\bbE |\widetilde{X}(s_J)|}{\lceil e^{\frac{\alpha}{2}J} \rceil} \le \frac{e^{ J^\gamma s_J}}{e^{\frac{\alpha}{2}J}}= e^{-\varepsilon J}.
\end{equation}
 $I_1$ can be estimated using the conditional expectation 
\begin{multline*}
	I_1 \le \bbE\left[ \textbf{1}_{\{|\widetilde{X}(s_J)|<\lceil e^{\frac{\alpha}{2}J} \rceil \} } \sum_{r \in \widetilde{X}(s_J)} \textbf{1}_{\{ Y_{r}(s_J)>e^J \}} \right] \\=  \bbE\left[ \textbf{1}_{\{|\widetilde{X}(s_J)|<\lceil e^{\frac{\alpha}{2}J} \rceil \} } \sum_{r \in \widetilde{X}(s_J)} \bbE \left[\left.\textbf{1}_{\{ Y_{r}(s_J)>e^J \}}\right| |\widetilde{X}(s_J)|\right] \right].
\end{multline*}
Let us observe that the conditional expectation presented above does not depend on $r$. It comes from the following fact: if at some time $\tau$ a particle $r$ produces another particle $\bar{r}$, then at time $t$ they are indistinguishable. Thus we obtain

\begin{align}
	I_1 \leq \lceil e^{\frac{\alpha}{2}J} \rceil \Pro (Y(s_J)>e^J) \label{5.4}
\end{align}
where $Y$ has distribution as in Lemma \ref{ruch}.
 Lemma \ref{ruch} and \eqref{5.4}  imply that for an arbitrary $\eta>0$ and large enough $J$

\begin{equation} 
I_1 \leq \lceil e^{\frac{\alpha}{2}J} \rceil \left(e^{-J(\alpha-\eta)}\cdot J^{1+\alpha}+ q^{J} \right). \label{oszacowanieI1}
\end{equation}

From \eqref{suma}, \eqref{oszacowanieI2} and \eqref{oszacowanieI1} we have
$$
\Pro\left(M(s_J)>e^J\right) \le e^{\lceil \frac{\alpha}{2}J \rceil} \left( e^{-J(\alpha-\eta)}\cdot J^{1+\alpha}+ q^J \right) +e^{-\varepsilon J}.
$$
Obviously $\sum e^{-\varepsilon J} < \infty$.  If $\eta < \alpha$ and $q$ is such that $e^{\alpha/2} \cdot q<1$ then also \\ $\sum  e^{\lceil \frac{\alpha}{2}J \rceil} \left( e^{-J(\alpha-\eta)}\cdot J^{1+\alpha}+  q^J \right) <\infty$, and the proof of Lemma is concluded. 
\end{proof}
We are ready to prove the bound from above in Theorem \ref{glowne}. The Borel-Cantelli Lemma and Lemma \ref{ostatnie} imply that  $M(s_J)<e^J$ for $J$ sufficiently large a.s. By analogous argument as in Section \ref{rozd} we obtain 
$$\liminf_{t \rightarrow \infty} \frac{\log M(t)}{t^\wy} \leq \left( \frac{1}{\frac{\alpha}{2} -\varepsilon } \right)^\wy \ a.s.$$
Since $\varepsilon>0$ was arbitrary, the upper bound in Theorem \ref{glowne} is obtained.

\textbf{Acknowledgment}: I would like to thank dr Piotr Mi{\l{}}o\'s and Katarzyna Kmieć for all comments which improved the presentation of the proof.

\noindent
Rafa{\l} Meller\\
Institute of Mathematics\\
University of Warsaw\\
Banacha 2, 02-097 Warszawa, Poland\\
{\tt r.meller@mimuw.edu.pl}

\end{document}